\documentclass{amsart}

\usepackage[english]{babel}
\usepackage[utf8x]{inputenc}
\usepackage[T1]{fontenc}

\usepackage{amsmath,amssymb,tikz-cd}
\usepackage{amsthm}
\usepackage{graphicx}
\usepackage[colorinlistoftodos]{todonotes}
\usepackage[colorlinks=true, allcolors=blue]{hyperref}
\usepackage{enumerate}
\usepackage{tensor}

\newcommand{\co}{\colon\thinspace}
\newcommand{\reg}{\text{reg}}
\newcommand{\id}{\text{id}}
\renewcommand{\reg}{\text{reg}}

\DeclareMathOperator{\dom}{dom}
\DeclareMathOperator{\im}{im}
\DeclareMathOperator{\At}{At}
\DeclareMathOperator{\gr}{gr}

\newtheorem{theorem}{Theorem}
\newtheorem{lemma}{Lemma}
\newtheorem{proposition}{Proposition}
\newtheorem{corollary}{Corollary}

\theoremstyle{definition}
\newtheorem{definition}{Definition}
\newtheorem{remark}{Remark}
\newtheorem{example}{Example}
\newtheorem{examples}{Examples}

\title{Suborbifolds and groupoid embeddings}

\author{Jo\~ao Nuno Mestre}
\address{Max-Planck-Institut f\"ur Mathematik, Vivatsgasse 7, 53111 Bonn, Germany}
\email{jnmestre@gmail.com}

\author{Martin Weilandt}
\address{Universidade Federal de Santa Catarina, Departamento de Matem\'atica, Campus Universit\'ario Trindade, Florian\'opolis-SC, 88040-900, Brazil}
\curraddr{mi Solutions \& Consulting GmbH, Frankfurter Stra\ss e 1, 65239 Hochheim am Main, Germany}
\email{martin.weilandt@alumni.hu-berlin.de}

\begin{document}

\begin{abstract}
Given the notion of suborbifold of the second author (based on ideas of Borzellino/Brunsden) and the classical correspondence (up to certain equivalences) between (effective) orbifolds via atlases and effective orbifold groupoids, we analyze which groupoid embeddings correspond to suborbifolds and give classes of suborbifolds naturally leading to groupoid embeddings.
\end{abstract}

\maketitle

\section{Introduction}
The study of the geometry of subspaces demands an appropriate notion of embeddings. Given the variety of definitions of smooth maps between (effective) orbifolds given by atlases (compare Satake's different notions of ``$C^\infty$-map'' in \cite{Satake1956} and \cite{Satake1957}, also see \cite{Chen2002, Borzellino2008}) and the variety of reasonable definitions of suborbifolds in this class (\cite{Weilandt2017}), it is common to focus on the setting of orbifold groupoids with its natural notions of homomorphism (\cite{Pohl2017, Chen2006}) and embedding (\cite{Cho2013}). The link between these two settings is given by the classical correspondence between diffeomorphism classes of (effective) orbifolds via atlases and Morita equivalence classes of orbifold groupoids of \cite{Moerdijk1997, Moerdijk2003}, which we recall briefly at the beginnings of Sections \ref{sec:groupoid-orbifold} and \ref{sec:orbifold-groupoid}.

This paper is structured as follows: In Section \ref{sec:prelim} we recall basic definitions of groupoids and (effective) orbifolds and consider certain classes of subgroupoids and suborbifolds. In Section \ref{sec:groupoid-orbifold} we illustrate how a subgroupoid of an effective orbifold groupoid naturally leads to a suborbifold. In Section \ref{sec:orbifold-groupoid} we give sufficient criteria for a suborbifold to lead to a subgroupoid and verify that certain graphs including the diagonal in any orbifold fall into this category.

It seems reasonable to believe that a closer look at \cite{Pohl2017} could lead to shorter proofs of our results, but we prefer to give a self-contained account based only on the classical correspondence mentioned above instead of working with the rather complex constructions from \cite{Pohl2017}.

All manifolds and all group actions in this work are assumed to be smooth ($C^\infty$) and all submanifolds are embedded. Manifolds are usually second countable and Hausdorff, the only exception being the manifold $G_1$ of arrows in a Lie groupoid $G_1\rightrightarrows G_0$, see Definition \ref{def:groupoid}.


\section{Preliminaries}
\label{sec:prelim}
\subsection{Groupoids and group actions}

Before getting to the theory of orbifolds, we recall some basic theory of Lie groupoids and subgroupoids, (we refer to \cite{Moerdijk2003} for further detail) and illustrate how group actions lead to natural examples of certain subgroupoids.

\begin{definition}
\label{def:groupoid}
A \emph{Lie groupoid} is given by a tuple $(G_0,G_1,s,t,m,u,i)$, where 
\begin{itemize}
\item $G_0$ and $G_1$ are smooth manifolds, and $G_0$ is second countable and Hausdorff (but $G_1$ not necessarily);
\item $s,t\co G_1\to G_0$ are submersions, called the \emph{source} and \emph{target} maps;
\item the \emph{multiplication} 
$m\co G_1 \tensor[_s]{\times}{_t} G_1\to G_1$, the \emph{unit} $u\co G_0\to G_1$ and the \emph{inverse} $i\co G_1\to G_1$ are smooth maps.
\end{itemize}

We will also use the notations $1_p=u(p)$, $g^{-1}=i(g)$ and $gh=m(g,h)$. The \emph{structure maps} $s,t,m,u,i$ are required to satisfy the following axioms:

\begin{itemize}
\item $s(gh)=s(h)$ and $t(gh)=t(g)$
\item $(gh)k=g(hk)$
\item $s(1_p)=t(1_p)=p$ and $g1_{s(g)}=1_{t(g)}g=g$
\item $s\circ i=t$, $t\circ i = s$, $gg^{-1}=1_{t(g)}$ and $g^{-1}g=1_{s(g)}$
\end{itemize}

We say that $G$ is a Lie groupoid \emph{over} $G_0$, and we often denote the groupoid $G$ by $G_1\rightrightarrows G_0$. The manifolds $G_0$ and $G_1$ are called the space of \emph{objects} and space of \emph{arrows} of $G$, respectively. Given $g\in G_1$ with $s(g)=p$ and $t(g)=q$, we will sometimes write $g\co p\to q$.
\end{definition}

\begin{examples}
\label{ex:groupoids}
\begin{enumerate}

\item \label{ex:groupoids:1} Any Lie group $G$ can be seen as a Lie groupoid over a point $G \rightrightarrows \{*\}$, with $m$ and $i$ given by multiplication and inverse in $G$, respectively.
\item Any smooth (Hausdorff, second countable) manifold $M$ can be considered as a Lie groupoid where the only arrows are unit arrows, $M\rightrightarrows M$.
\item Given a smooth action of a Lie group $G$ on a manifold $M$, we can construct the action groupoid $G\ltimes M = (G\times M\rightrightarrows M)$, with structure maps determined by $s(g,p)=p$, $t(g,p)=gp$, $(g,hq)(h,q)=(gh,q)$.

\end{enumerate}
\end{examples}

\begin{definition} Let $G$ be a Lie groupoid.
Two points $p,q\in G_0$ are \emph{equivalent} if and only if there is an arrow $g\in G_1$ such that $s(g)=p$ and $t(g)=q$. The equivalence class of $p\in G_0$ with respect to this equivalence relation is called the \emph{orbit} of $p$, and denoted by $Gp$. The quotient topological space of $G_0$ with respect to this equivalence relation is called the \emph{orbit space} of $G$, and is denoted by $|G|$.

Given $p,q\in G_0$, we write $G(p,q)$ for the set of all arrows $p\to q$. Given $p\in G_0$, we write $G_p$ for the group $G(p,p)$ (with composition given by $m$).
\end{definition}

\begin{remark} Given a Lie groupoid $G$, the projection $\pi\co G_0\to |G|$ is an open map. Indeed, given an open subset $U$ of $G_0$, its image $\pi(U)$ is open in $|G|$ if and only if $\pi^{-1}(\pi(U))$ is open in $G_0$. That is the case, since $\pi^{-1}(\pi(U))= t(s^{-1}(U))$, and $s$ and $t$ are submersions.
\end{remark}

\begin{definition}
Let $G,H$ be Lie groupoids. 
\begin{enumerate}
\item A \emph{Lie groupoid morphism} $\Phi\co H\to G$ is a pair $(\Phi_0,\Phi_1)$ of smooth maps $\Phi_i\co H_i\to G_i$, $i=0,1$, which respect all structure maps, i.e., $s^G\circ\Phi_1=\Phi_0\circ s^H$, $t^G\circ\Phi_1=\Phi_0\circ t^H$, $\Phi_1\circ m^H=m^G\circ(\Phi_1\times\Phi_1)_{|H_1\tensor[_s]{\times}{_t} H_1}$, $u^G\circ\Phi_0=\Phi_1\circ u^H$, $i^G\circ\Phi_1=\Phi_1\circ i^H$.
\item A \emph{Lie groupoid embedding} is a Lie groupoid morphism $\Phi\co H\to G$ such that each $\Phi_i\co H_i\to G_i$, $i=0,1$, is a smooth embedding between manifolds.
\end{enumerate}
  
\end{definition}

\begin{definition}
Given a Lie groupoid morphism $\Phi\co H\to G$, the induced continuous map $|H|\ni Hp\mapsto G\Phi(p)\in |G|$ between the orbit spaces will be denoted by $|\Phi|$.
\end{definition}

\begin{lemma}
\label{lemma:essinj}
  Let $\Phi\co H\to G$ be a Lie groupoid embedding. Then the following conditions are equivalent.
  \begin{enumerate}
  \item \label{essinj1} For every $p,q\in H_0$ the existence of an arrow $\Phi_0(p)\to\Phi_0(q)$ in $G_1$ implies the existence of an arrow $p\to q$ in $H_1$.
  \item \label{essinj2} $|\Phi|\co |H|\to |G|$ is injective.
  \item \label{essinj3} For every $p\in H_0$ we have $\Phi_0(Hp)=G\Phi_0(p)\cap \Phi_0(H_0)$.
  \end{enumerate}
\end{lemma}

\begin{proof}
  To see that \eqref{essinj1} implies \eqref{essinj2}, let $p,q\in H_0$ such that $G\Phi_0(p)=G\Phi_0(q)$. Since $\Phi_0(p)$ and $\Phi_0(q)$ are in the same $G$-orbit, then by \eqref{essinj1} $p$ and $q$ are in the same $H$-orbit and hence $Hp=Hq$.
  
  To see that \eqref{essinj2} implies \eqref{essinj3}, let $p\in H_0$. Since $\Phi_0$ is a Lie groupoid morphism, the inclusion $\Phi_0(Hp)\subset G\Phi_0(p)\cap \Phi_0(H_0)$ automatically holds. For the other inclusion let $q\in H_0$ such that $\Phi_0(q)\in G\Phi_0(p)$. Then $\Phi_0(p)$ and $\Phi_0(q)$ lie in the same $G$-orbit and, by \eqref{essinj2}, $p$ and $q$ lie in the same $H$-orbit. Then $q\in Hp$ and hence $\Phi_0(q)\in\Phi_0(Hp)$.
  
  To see that \eqref{essinj3} implies \eqref{essinj1}, note that the existence of an arrow $\Phi_0(p)\to\Phi_0(q)$ means $\Phi_0(q)\in G\Phi_0(p)\cap\Phi_0(H_0)$. Since $\Phi_0$ is injective, \eqref{essinj3} implies $q\in Hp$.
\end{proof}

We should note that the hypothesis of the lemma above can be weakened to asking that $\Phi$ be an injective groupoid morphism, since the smooth structure is not relevant for the proof. Furthermore, the equivalence of \eqref{essinj1} and \eqref{essinj2} above has already been observed in \cite{Cho2013} and does not require the injectiveness assumption on $\Phi$. The following terminology is borrowed from the same source.

\begin{definition}
  A Lie groupoid embedding is called \emph{essentially injective} if it satisfies one of the equivalent conditions from Lemma \ref{lemma:essinj}.
\end{definition}

We also recall the following standard definition.

\begin{definition}
  Let $\Phi\co H\to G$ be a Lie groupoid embedding. $\Phi$ is \emph{fully faithful} if ${\Phi_1}_{|H(p,q)}\co H(p,q)\to G(\Phi_0(p),\Phi_0(q))$ is a bijection, for every $p,q\in H_0$. 
\end{definition}

From condition \eqref{essinj1} in Lemma \ref{lemma:essinj} it is clear that every fully faithful Lie groupoid embedding is essentially injective. The converse does not hold in general. A very simple counterexample is given by considering any closed proper Lie subgroup of a Lie group, e.g., the identity subgroup of any non-trivial Lie group, compare Example \ref{ex:groupoids} \eqref{ex:groupoids:1}. (Counterexamples more relevant in the context of orbifolds can be obtained applying Proposition \ref{prop:subgroupoids} below to \cite[Examples 10, 11, 14]{Borzellino2015}.)

\begin{definition}
  Let $G=(G_1\rightrightarrows G_0)$ be a Lie groupoid. A \emph{Lie subgroupoid} of $G$ is a Lie groupoid $H=(H_1\rightrightarrows H_0)$ such that $H_i\subset G_i$ for $i=0,1$ and the inclusion $\iota\co H\hookrightarrow G$ is a Lie groupoid embedding. $H$ is called \emph{essentially injective} if $\iota$ is essentially injective and \emph{full} if $\iota$ is fully faithful.
\end{definition}

Note that the notion of Lie subgroupoid we consider is commonly called by embedded subgroupoid in the literature.

Group actions provide an easy way to construct essentially injective and/or full subgroupoids. First recall the following terminology from \cite{Weilandt2017}.
\begin{definition}
  Let $K$ be a Lie group acting on a manifold $M$, let $L\subset K$ be a closed subgroup and let $N\subset M$ be an $L$-invariant submanifold.
  \begin{enumerate}
  \item $N$ is an \emph{$L$-submanifold} if for every $k\in K,p\in N$ such that $kp\in N$ there is $l\in L$ such that $lp=kp$. If, moreover, the action of $L$ on $N$ is effective, $N$ is called an \emph{effective $L$-submanifold}.
  \item $N$ is a \emph{full $L$-submanifold} if for every $k\in K, p\in N$ such that $kp\in N$ we have $k\in L$. 
  \end{enumerate}
\end{definition}

\begin{proposition}
\label{prop:subgroupoids}
  Let $K$ be a Lie group acting on a manifold $M$, let $L\subset K$ be a closed subgroup and let $N\subset M$ be an $L$-invariant submanifold. Then:
  \begin{enumerate}
    \item \label{prop:subgroupoids:1} $L\ltimes N$ is a Lie subgroupoid of $K\ltimes M$.
    \item \label{prop:subgroupoids:2} $L\ltimes N$ is essentially injective if and only if $N$ is an $L$-submanifold of $M$.
    \item \label{prop:subgroupoids:3} $L\ltimes N$ is full if and only if $N$ is a full $L$-submanifold of $M$.
  \end{enumerate}
\end{proposition}

\begin{proof}
  The proof of \eqref{prop:subgroupoids:1} is straightforward. \eqref{prop:subgroupoids:2} follows directly from Lemma \ref{lemma:essinj} \eqref{essinj1} and the definition of an $L$-submanifold above.
  
  To see \eqref{prop:subgroupoids:3} first note that fullness of $L\ltimes N$ implies that $N$ is an $L$-submanifold (by \eqref{prop:subgroupoids:2}) and $L_p=(L\ltimes N)_p=(K\ltimes M)_p=K_p$ for every $p\in N$. These two facts are easily seen to imply that $N$ is a full $L$-submanifold (see \cite[Lemma 2.4]{Weilandt2017}). Vice versa, if $N$ is a full $L$-submanifold, then given an arrow $g=(k,p)\in (K\ltimes M)_1$ from $p\in N$ to $q\in N$, we note that $kp=q$ implies $k\in L$ and hence $g\in(L\ltimes N)_1$.
\end{proof}

\subsection{Orbifolds and suborbifolds}
In this section we summarize the definitions of orbifold groupoids and (effective) orbifolds (via charts) and give characterizations of suborbifolds which will come in handy in the following sections.
First recall the following classical definition.
\begin{definition}

A  Lie groupoid $G$ is called \emph{proper} if $G_1$ is Hausdorff and the map $(s,t)\co G_1\to G_0\times G_0$ is proper. It is called \emph{\'etale} if $s,t\co G_1\to G_0$ are local diffeomorphisms.
An \emph{orbifold groupoid} is a Lie groupoid which is proper and \'etale.
  
  An orbifold groupoid $G$ is called \emph{effective} if the map $G_p\to \text{Diff}_p$ taking each arrow $g\in G_p$ to the induced germ of diffeomorphisms around $p$ fixing $p$ is injective. 
\end{definition}

\begin{example}
  Let $G$ be an orbifold groupoid and let $H$ be a closed Lie subgroupoid (i.e., $H_1$ is closed in $G_1$). Then $H$ is also an orbifold groupoid: the maps $s^H,t^H\co H_1\to H_0$ are local diffeomorphisms, since they are the restrictions of $s,t$, respectively. Since $H_1$ is closed in $G_1$, the restriction $(s^H, t^H)$ of $(s,t)$ to $H_1$ is proper. 
\end{example}

\begin{remark}
  We should note that an orbifold groupoid is often defined to be any Lie groupoid which is Morita equivalent to a proper \'etale groupoid (see for example \cite{Moerdijk2002}), but such a definition is essentially equivalent to the notion used in this text.
\end{remark}

We now recall the definition of an (effective) orbifold in terms of charts (compare \cite{Chen2002, Moerdijk1997, Moerdijk2003}).

\begin{definition}
  Let $X$ be a second countable Hausdorff space. An $n$-dimensional orbifold \emph{chart} on $X$ is a quadruple $(U,\widetilde{U}/\Gamma,\pi)$ in which $U\subset X$ is open, $\widetilde{U}$ is an $n$-dimensional connected manifold, $\Gamma$ is a finite group acting smoothly and effectively on $\widetilde{U}$ and $\pi\co \widetilde{U}\to U$ is a continuous $\Gamma$-invariant map which induces a homeomorphism $\overline{\pi}\co\widetilde{U}/\Gamma\to U$. Given charts $(U_i,\widetilde{U}_i/\Gamma_i,\pi_i)$, $i=1,2$, such that $U_1\subset U_2$, an \emph{injection} from $\pi_1$ to $\pi_2$ is a smooth embedding $\lambda\co\widetilde{U}_1\to\widetilde{U}_2$ such that $\pi_2\circ\lambda=\pi_1$. Two charts $(U_i,\widetilde{U}_i/\Gamma_i,\pi_i)$, $i=1,2$, are \emph{compatible} if for every $x\in U_1\cap U_2$ there is a chart $(W,\widetilde{W}/K,\sigma)$ such that $x\in W\subset U_1\cap U_2$ and, for each $i=1,2$, an injection $\lambda_i\co\widetilde{W}\to\widetilde{U}_i$ from $\sigma$ to $\pi_i$. An orbifold \emph{atlas} on $X$ is a collection $\{(U_\alpha,\widetilde{U}_\alpha/\Gamma_\alpha,\pi_\alpha)\}_\alpha$ of compatible charts on $X$ such that $X=\bigcup_\alpha U_\alpha$.
  
  An $n$-dimensional \emph{orbifold} is a pair $(X,\mathcal{A})$ of a second countable Hausdorff space $X$ and a maximal orbifold atlas $\mathcal{A}$.
  
  Given a point $x$ in an orbifold $(X,\mathcal{A})$, its \emph{isotropy} is the isomorphism class of $\Gamma_{\widetilde{x}}$ for some (hence every) chart $(U,\widetilde{U}/\Gamma,\pi)$ in $\mathcal{A}$ and $\widetilde{x}\in\pi^{-1}(x)$. $x$ is \emph{regular} if its isotropy is trivial, otherwise $x$ is called \emph{singular}. By $(X,\mathcal{A})^\reg$ we will denote the (open and dense) subset of regular points in an orbifold $(X,\mathcal{A})$.
  
  Given two orbifolds $(X,\mathcal{A})$, $(X^\prime,\mathcal{A}^\prime)$, the \emph{product orbifold} is the set $X\times X^\prime$ together with the maximal atlas (denoted by $\mathcal{A}\times\mathcal{A}^\prime$) containing all charts of the form $(U\times U^\prime,(\widetilde{U}\times\widetilde{U}^\prime)/(\Gamma\times\Gamma^\prime),\pi\times\pi^\prime)$ where $(U,\widetilde{U}/\Gamma,\pi)\in\mathcal{A}$ and $(U^\prime,\widetilde{U}^\prime/\Gamma^\prime,\pi^\prime)\in\mathcal{A}^\prime$.
\end{definition}

\begin{remark}
  Note that, following \cite{Chen2002}, we allow that each chart domain $\widetilde{U}$ is a manifold. Choosing these domains sufficiently small as in \cite[Lemma 4.1.1]{Chen2002}, we could assume that they are open subsets of $\mathbb{R}^n$ and obtain an equivalent (and apparently more common) orbifold definition.
  
  We should also emphasize that the term ``orbifold groupoid'' does not include effectiveness conditions, whereas every ``orbifold'' (sometimes referred to as ``reduced'' or ``effective'' orbifold) is assumed to be equipped with effective group actions as above. (Note, however, that there is a notion of ``ineffective orbifold'' using charts which corresponds to the class of orbifold groupoids \cite{Pronk2016}.)
\end{remark}

\begin{definition}
  \label{def:suborbi-cover}
  Let $\mathcal{A}=\{(U_i,\widetilde{U}_i/\Gamma_i,\pi_i)\}_{i\in I}$ be an $n$-dimensional orbifold atlas on a second countable Hausdorff space $X$ and let $Y\subset X$ be a subset. A $k$-dimensional \emph{suborbifold cover} on $Y\subset X$ with respect to $\mathcal{A}$ is a subset $J\subset I$ together with a family $\{\widetilde{V}_j\}_{j\in J}$ of connected $k$-dimensional manifolds such that
  \begin{enumerate}
  \item \label{def:suborbi-cover1} for every $j\in J$ there is a subgroup $\Delta_j$ of $\Gamma_j$ such that $\widetilde{V}_j$ is a $\Delta_j$-submanifold of $\widetilde{U}_j$ with the property that $\pi_j(\widetilde{V}_j)$ is open in $Y$,
  \item $\bigcup_j\pi_j(\widetilde{V}_j)=Y$.
  \end{enumerate}
  If each $\widetilde{V}_j$ is closed in $\widetilde{U}_j$, we call $\{\widetilde{V}_j\}_{j\in J}$ a \emph{closed cover}.
\end{definition}

Note that, by \cite[Proposition 3.3]{Weilandt2017} a closed suborbifold cover as above defines an orbifold atlas $\{(V_j,\widetilde{V}_j/(\Delta_j/K_j),{\pi_j}_{|\widetilde{V}_j})\}_{j\in J}$ on $Y$ (with $K_j\subset\Delta_j$ denoting the corresponding kernel).

\begin{definition}
 Given an orbifold atlas $\mathcal{A}$ on a second countable Hausdorff space $X$ and a closed suborbifold cover $\{\widetilde{V}_j\}$ on $Y\subset X$ as in Definition \ref{def:suborbi-cover}, the orbifold atlas $\{(V_j,\widetilde{V}_j/(\Delta_j/K_j),{\pi_j}_{|\widetilde{V}_j})\}_{j\in J}$ is called the atlas on $Y$ \emph{induced} by $\{\widetilde{V}_j\}_{j\in J}$ (and $\mathcal{A}$).
\end{definition}

\begin{remark}
  Strictly speaking, the groups $\Delta_j$ satisfying the conditions from Definition \ref{def:suborbi-cover} \eqref{def:suborbi-cover1} may not be unique. But since the choice of another $\Delta_j^\prime$ would just induce an equivalent effective action (by $\Delta_j^\prime/K_j^\prime$) on $\widetilde{V}_j$ (as follows from \cite[Corollary 3.10]{Lange2018}), we still refer to \emph{the} induced atlas in the definition above.
\end{remark}

\begin{definition}
A suborbifold cover as in Definition \ref{def:suborbi-cover} is \emph{embedded} if each $\Delta_j$ can be chosen to act effectively on $\widetilde{V}_j$. It is called \emph{fully embedded} if each $\Delta_j$ can be chosen such that $\widetilde{V}_j$ is a full effective $\Delta_j$-submanifold.
\end{definition}

\begin{remark}
By \cite[Remark 3.2]{Weilandt2017} the existence of a (embedded) suborbifold cover implies the existence of a closed (embedded) suborbifold cover. By the same argument, an analogous statement holds for a fully embedded cover.
\end{remark}

The first two items of the definition below are just \cite[Definition 3.1]{Weilandt2017}, now using the suborbifold cover terminology introduced above.

\begin{definition}
\begin{enumerate}
  \item A \emph{suborbifold} of an orbifold $(X,\mathcal{A})$ is a subset of $X$ which admits a suborbifold cover with respect to $\mathcal{A}$.
  \item An \emph{embedded suborbifold} of an orbifold $(X,\mathcal{A})$ is a subset of $X$ which admits an embedded suborbifold cover with respect to $\mathcal{A}$.
   \item A \emph{fully embedded suborbifold} of an orbifold $(X,\mathcal{A})$ is a subset of $X$ which admits a fully embedded suborbifold cover with respect to $\mathcal{A}$.

\end{enumerate}
\end{definition}

\begin{remark}
In this text we introduce the notion of fully embedded suborbifolds instead of working with ``full'' suborbifolds as in \cite{Weilandt2017} (inspired by the homonymous notion in \cite{Borzellino2012}) to guarantee that isotropy is preserved. For instance, a single point in an orbifold is always ``full'' (and embedded) but it is only fully embedded if it is regular in the ambient orbifold.
\end{remark}

\cite[Proposition 3.3]{Weilandt2017} shows that every suborbifold carries a canonical orbifold structure, for which we introduce the following notation.

\begin{definition}
\label{def:suborbi-atlas}
  Given a suborbifold $Y$ of an orbifold $(X,\mathcal{A})$, the maximal orbifold atlas containing all atlases on $Y$ induced by closed suborbifold covers with respect to $\mathcal{A}$ will be denoted by $\mathcal{A}_{|Y}$.
\end{definition}

\section{Subgroupoids leading to suborbifolds}
\label{sec:groupoid-orbifold}
Following \cite{Moerdijk1997}, to an effective orbifold groupoid $G$ one can associate a canonical orbifold atlas on $|G|$ : Let $x=Gp\in |G|$. By the proof of the implication $(4)\Rightarrow (1)$ in \cite[Theorem 4.1]{Moerdijk1997}, each $g\in G_p$ has an open neighborhood $O_g$ in $G_1$ such that $s_{|O_g}$ and $t_{|O_g}$ are embeddings and there is a connected open neighborhood $N_p$ of $p$ such that 
  \begin{equation}
  	(s,t)^{-1}(N_p\times N_p)=\coprod_{g\in G_p} O_g
  	\label{eq:disj}
  \end{equation}
  and $N_p$ is invariant under the effective $G_p$-action given by $g\cdot q:=\tilde{g}(q):=t\circ(s_{|O_g})^{-1}(q)$. Note that condition \eqref{eq:disj} implies that two points in $N_p$ are in the same $G_p$-orbit if and only if they are in the same $G$-orbit. Using the canonical projection $\pi_p\co N_p\to N_p/G_p$, we obtain an orbifold chart $(\pi_p(N_p), N_p/G_p,\pi_p)$ on $|G|$ around $x$.
  
  It has been shown in \cite{Moerdijk1997} that any two charts as above are compatible. In particular, the maximal atlas on $|G|$ containing the charts $(\pi_p(N_p), N_p/G_p,\pi_p)$, $p\in G$, does not depend on the choice of concrete $O_g$, $N_p$ with the properties above. We shall denote this atlas by $\At(G)$. Applying the construction above to subgroupoids, we obtain the following theorem. (For item \eqref{th:groupoid-orbifold3} also recall Definition \ref{def:suborbi-atlas}.)

\begin{theorem}
\label{th:groupoid-orbifold}
  Let $H$ be a Lie subgroupoid of an effective orbifold groupoid $G$ such that $|\iota|\co |H|\to |G|$ is a topological embedding. Then:
  \begin{enumerate}
  \item \label{th:groupoid-orbifold1} The orbit space $|H|$ is a suborbifold of $(|G|,\At(G))$.
  \item \label{th:groupoid-orbifold3} If $H$ is an effective orbifold groupoid, then $\At(G)_{||H|}=\At(H)$
  and $|H|$ is an embedded suborbifold of $(|G|,\At(G))$.   If, moreover, $H$ is full in $G$, then $|H|$ is a fully embedded suborbifold of $(|G|,\At(G))$.
  \end{enumerate}
\end{theorem}

\begin{proof}
  To see \eqref{th:groupoid-orbifold1}, first let $x=Hp\in |H|$ and consider the chart $(N_p/G_p, N_p/G_p,\pi_p)$ of $|G|$ around $x$ as above. Now consider $H_p=H_1\cap G_p$ and let $S_p$ denote the connected component of $H_0\cap N_p$ containing $p$. Then $S_p$ is an $H_p$-invariant submanifold of $N_p$. To see that it is an $H_p$-submanifold, let $q\in S_p$ and $g\in G_p$ such that $\tilde{g}(q)\in S_p$. Since $q,\tilde{g}(q)\in N_p$, we have $\tilde{g}(q)\in Gq$. Since $H$ is essentially injective in $G$, we obtain $\tilde{g}(q)\in H_0\cap Gq=Hq$. This means that there is an arrow $h \in H_1$ from $q$ to $\tilde{g}(q)$ and hence $\tilde{h}(q)=\tilde{g}(q)$.
  
  To finish our argument for \eqref{th:groupoid-orbifold1}, we are left to show that $\pi_p(S_p)$ is an open subset of $|H|$. It obviously is a subset. To see that it is open, note that the projection $\sigma\co H_0\to |H|$ of the Lie groupoid $H$ is open
and hence $\sigma^{-1}(\pi_p(S_p))=\sigma^{-1}(\sigma(S_p))$ is open in $H_0$.

  To see \eqref{th:groupoid-orbifold3}, note that for each $p\in H_0$ and $h\in H_p$ we have
  $(s^H,t^H)^{-1}(S_p\times S_p)=\coprod_{h\in H_p} P_h$ with $P_h \subset O_h\cap H_1$ an open neighborhood of $h$ in $H_1$ and that $S_p$ is invariant under the $H_p$-action given by $h\circ q:=\tilde{h}(q):=t^H\circ(s^H_{|P_h})^{-1}(q)$, which coincides with the restriction of the $G_p$-action on $N_p$ to $S_p$. Diminishing $N_p$, we can assume that $S_p$ is closed and that $H_p$ acts effectively on $S_p$. Hence, $(\pi_p(S_p),S_p/H_p,{\pi_p}_{|S_p})$ lies both in $\At(H)$ and ${\At(G)}_{||H|}$. Maximality implies the first statement of \eqref{th:groupoid-orbifold3}.
  
  For the full case in \eqref{th:groupoid-orbifold3} note that if $H$ is a full subgroupoid, then $H_p=G_p$ and hence $S_p$ is a full $H_p$-submanifold of $N_p$.
\end{proof}

As a rather simple application of the theorem above we can give an alternative proof of a special case of \cite[Corollary 4.6]{Weilandt2017}.

\begin{corollary}
  Let $K$ be a discrete group acting smoothly, properly and effectively on a manifold $M$. If $L$ is a subgroup of $K$, $N$ is an $L$-submanifold of $M$ and $N/L\ni Lp\mapsto Kp\in M/K$ is a topological embedding, then $N/L$ is a suborbifold of $M/K$.
\end{corollary}

\begin{proof}
  Since $|K\ltimes M|=M/K$ and $|L\ltimes N|=N/L$, the result follows from part \eqref{th:groupoid-orbifold1} of Theorem \ref{th:groupoid-orbifold} .
\end{proof}

\section{Suborbifolds leading to subgroupoids}
\label{sec:orbifold-groupoid}
Given a countable orbifold atlas $\mathcal{A}=\{(U_i,\widetilde{U}_i/\Gamma_i,\pi_i)\}_{i\in I}$ on a second countable Hausdorff space $X$, we consider the groupoid $\mathcal{G}(\mathcal{A})$ given by the following construction (see \cite{Moerdijk1997}):

Set $\mathcal{G}(\mathcal{A})_0=\coprod_{i\in I}\widetilde{U}_i$. A \emph{transition} (on $\mathcal{G}(\mathcal{A})_0$) is a diffeomorphism between open subsets of $\mathcal{G}(\mathcal{A})_0$. Writing $\pi=\coprod_{i\in I}\pi_i$ and denoting by $P(\mathcal{A})$ the pseudogroup of transitions $f$ which satisfy $\pi\circ f=\pi_{|\dom f}$, set \[ \mathcal{G}(\mathcal{A})_1=\{[f]_p;~f\in P(\mathcal{A}),~p\in\dom f\},\] where $[f]_p$ denotes the germ of $f$ at $p$. The topology and smooth structure on $\mathcal{G}(\mathcal{A})_1$ are determined by requiring that, for each $f\in P(\mathcal{A})$, the bijection $\dom f\ni p\mapsto [f]_p\in\{[f]_q;~q\in\dom f\}$ be a diffeomorphism. The structure maps are given by $s([f]_p)=p$, $t([f]_p)=f(p)$, $m([g]_{f(p)},[f]_p)=[g\circ f]_p$, $u(p)=[\id]_p$, $i([f]_p)=[f^{-1}]_{f(p)}$.
With these definitions $\mathcal{G}(\mathcal{A})$ becomes an effective orbifold groupoid.
By $\varepsilon_\mathcal{A}$ we denote the homeomorphism $|\mathcal{G}(\mathcal{A})|\ni\mathcal{G}(\mathcal{A})p\mapsto \pi(p)\in X$.

Recall from \cite{Moerdijk2003} that, up to certain equivalences, the construction above can be seen as inverse to the construction of an orbifold out of an effective orbifold groupoid given in Section \ref{sec:groupoid-orbifold}: Given an orbifold $(X,\mathcal{A})$, the orbifold
$(|\mathcal{G}(\mathcal{A})|,\At(\mathcal{G}(\mathcal{A})))$ is diffeomorphic to $(X,\mathcal{A})$. On the other hand, given an effective orbifold groupoid $G$, the groupoid $\mathcal{G}(\At(G))$ is Morita equivalent to $G$. Since we will not work with the latter concept, we refer the reader to \cite{Moerdijk2003} for more details.

Now consider a closed suborbifold cover $\{\widetilde{V}_j\}_{j\in J}$ on a set $Y$ with respect to a countable orbifold atlas $\mathcal{A}$ on some second countable Hausdorff space $X$ and denote the induced orbifold atlas on $Y$ by $\mathcal{B}$. We would like to construct a certain Lie groupoid embedding $\mathcal{G}(\mathcal{B})\to\mathcal{G}(\mathcal{A})$.

\begin{definition}
\label{def:strongcover}
  Let $X$ be a second countable Hausdorff space, let $\mathcal{A}$ be a countable orbifold atlas on $X$ and let $Y$ be a subset of $X$. Let $\{\widetilde{V}_j\}_{j\in J}$ be a closed suborbifold cover on $Y$ with respect to $\mathcal{A}$. $\{\widetilde{V}_j\}_{j\in J}$ is called \emph{strong} if, with $\mathcal{B}$ denoting the induced orbifold atlas on $Y$, there is a Lie groupoid embedding $\Phi\co \mathcal{G}(\mathcal{B})\to\mathcal{G}(\mathcal{A})$ such that (with $Y\hookrightarrow X$ denoting the inclusion) the following diagram commutes.

\begin{center}
\begin{tikzcd}
\lvert\mathcal{G}(\mathcal{B})\rvert \arrow[r, "\lvert\Phi\rvert"] \arrow[d, "\varepsilon_\mathcal{B}"'] & \lvert\mathcal{G}(\mathcal{A})\rvert \arrow[d, "\varepsilon_\mathcal{A}"] \\
X \arrow[r, hook]                                                                    & Y
\end{tikzcd}
\end{center}

\end{definition}

\begin{examples}
\label{ex:strong}
\begin{enumerate}
  \item \label{ex:strong:1} Every closed embedded suborbifold cover $\{\widetilde{V}_1\}$ (i.e. consisting of a single element) with respect to a countable orbifold atlas $\mathcal{A}=\{(U_i,\widetilde{U}_i/\Gamma_i,\pi_i)\}_{i\in I}$ is strong: Letting $\mathcal{B}=\{(V_1,\widetilde{V}_1/\Delta_1,{\pi_1}_{|\widetilde{V}_1})\}$ (with $\Delta_1\subset \Gamma_1$) denote the induced atlas on the suborbifold, arrows in $\mathcal{G}(\mathcal{B})$ correspond to elements of $\Delta_1$ (as follows directly from \cite[Lemma 2.11]{Moerdijk2003}) and hence the inclusion $\Delta_1\hookrightarrow \Gamma_1$ induces a canonical embedding $\Phi_1\co\mathcal{G}(\mathcal{B})_1\to\mathcal{G}(\mathcal{A})_1$. Together with the inclusion $\Phi_0\co\mathcal{G}(\mathcal{B})_0=\widetilde{V}_1\hookrightarrow\coprod_{i\in I}\widetilde{U}_i=\mathcal{G}(\mathcal{A})_0$, we obtain a Lie groupoid embedding $\Phi$ satisfying $\varepsilon_\mathcal{A}\circ |\Phi|=\varepsilon_\mathcal{B}$.
  
  \item The following example illustrates that the groupoid embedding $\Phi$ in Definition \ref{def:strongcover} is not necessarily unique: We can consider an open disk $\widetilde{U}$ in $\mathbb{R}^2$ with center $(0,0)$ and the rotation $\rho$ by $\pi$. Setting $\Gamma=\langle\rho\rangle\simeq\mathbb{Z}_2$, we obtain the orbifold $U=\widetilde{U}/\Gamma$. Taking $\mathcal{A}$ to consist of two copies $\widetilde{U}_i\to U_i$, $i=1,2$ of the chart given by the canonical projection $\widetilde{U}\to U$ and $\widetilde{V}_1=\widetilde{V}_2=\{(0,0)\}$, we obtain a strong suborbifold cover on $V=\widetilde{V}_1/\{\id\}\subset U$ with respect to $\mathcal{A}$. Letting $\Phi_0\co \coprod_i\widetilde{V}_i\to \coprod_i\widetilde{U}_i$ denote the inclusion, we can define $\Phi_1$ by sending $\id\co\widetilde{V}_1\to\widetilde{V}_2$ either to the identity $\widetilde{U}_1\to\widetilde{U}_2$ or to $\rho$ and obtain two different groupoid embeddings $\Phi$ as in Definition \ref{def:strongcover}.
\end{enumerate}
\end{examples}

\begin{remark}
\begin{enumerate}
\item The commutative diagram in the definition above directly implies that $\Phi$ is essentially injective.
\item The term ``strong'' refers to the idea of a ``strong map'' of orbifolds in \cite{Moerdijk1997}.
\end{enumerate}
\end{remark}

\begin{definition}
\label{def:strong-suborbifold}
  Let $(X,\mathcal{A})$ be an orbifold. A subset $Y$ of $X$ is a \emph{strong suborbifold} if there is a strong suborbifold cover $\{\widetilde{V}_j\}_{j\in J}$ on $Y$ with respect to some countable atlas contained in $\mathcal{A}$.
\end{definition}

\begin{example}
\label{ex:discrete}
  Every discrete subset of an orbifold is a zero-dimensional strong suborbifold: If $(X,\mathcal{A})$ is an orbifold and $Y\subset X$ is discrete, then $Y=\{x_1,x_2,\ldots\}$. For each $j$ let $(U_j,\widetilde{U}_j/\Gamma_j,\pi_j)\in\mathcal{A}$ such that $x_j\in U_j$ and complete $\{(U_j,\widetilde{U}_j/\Gamma_j,\pi_j)\}_{j\ge 1}$ to a countable atlas $\mathcal{A}^\prime\subset \mathcal{A}$ on $X$. Fixing $\widetilde{x}_j\in\pi_j^{-1}(x_j)$ and setting $\widetilde{V}_j:=\{\widetilde{x}_j\}$, the family $\{\widetilde{V}_j\}_{j\ge 1}$ is a closed suborbifold cover of $Y$ with respect to $\mathcal{A}^\prime$. Since every transformation on $\coprod_j\widetilde{V}_j$ is a restriction of the identity, we can just extend it to the identity on $\widetilde{U}_j$ to obtain an embedding $\Phi_1$. Together with $\Phi_0$ given by the inclusions $\widetilde{V}_j\hookrightarrow\widetilde{U}_j$, we obtain an embedding $\Phi\co\mathcal{G}(\mathcal{B})\to\mathcal{G}(\mathcal{A}^\prime)$ such that $\varepsilon_{\mathcal{A}'}\circ |\Phi|=\varepsilon_\mathcal{B}$.
\end{example}

The following proposition gives an alternative definition of a strong suborbifold.
\begin{proposition}
\label{prop:strong-suborbifold}
  Given an orbifold $(X,\mathcal{A})$, a subset $Y$ of $X$ is a strong suborbifold if and only if there is a suborbifold cover $\{\widetilde{V}^\prime_j\}_{j\in J}$ with respect to some countable atlas  $\mathcal{A}^\prime\subset\mathcal{A}$ and (with $\mathcal{B}^\prime$ denoting the orbifold atlas on $Y$ induced by $\{\widetilde{V}^\prime_j\}_{j\in J}$) a groupoid embedding $\Phi^\prime: \mathcal{G}(\mathcal{B}^\prime)\to\mathcal{G}(\mathcal{A^\prime})$ such that $\Phi^\prime_0$ is the inclusion.
\end{proposition}
\begin{proof}
  Given $\{\widetilde{V}_j\}_{j\in J}$ as in Definition \ref{def:strong-suborbifold} and an embedding $\Phi:\mathcal{G}(\mathcal{B})\to\mathcal{G}(\mathcal{A}^\prime)$ such that $\varepsilon_\mathcal{A^\prime}\circ |\Phi| = \varepsilon_\mathcal{B}$, it suffices to define $\widetilde{V}_j^\prime = \Phi_0(\widetilde{V}_j)$ and $\Phi^\prime_1([f]_p)=\Phi_1([f\circ\Phi_0]_{\Phi_0^{-1}(p)})$ for $[f]_p\in\mathcal{G}(\mathcal{B}^\prime)_1$. The other implication is clear.
\end{proof}

\begin{example}
\label{ex:nonstrong}
 An example of a suborbifold which is not strong is given by the (non-embeded) suborbifold from \cite[Example 12]{Borzellino2015}: If $M=\mathbb{C}^2$ and $G\simeq\mathbb{Z}_4$ is the group generated by $\begin{pmatrix}i&0\\0&-1\end{pmatrix}$, then $X=M/G$ is a good orbifold (equipped with the maximal atlas $\mathcal{A}$ containing the canonical global chart) and $Y=(\{0\}\times \mathbb{C})/G$ is a suborbifold. Assume that $Y$ is a strong suborbifold. Then there is a a countable atlas $\mathcal{A}^\prime=\{(U_i,\widetilde{U}_i/\Gamma_i,\pi_i)\}_{i\in I}\subset\mathcal{A}$ and a closed suborbifold cover $\{\widetilde{V}_j\}_{j\in J}$ with respect to $\mathcal{A}^\prime$ as in Proposition \ref{prop:strong-suborbifold}. Consider $j_0$ such that $[(0,0)]\in\pi_{j_0}(\widetilde{V}_{j_0})$ and let $\Delta_{j_0}$ be a subgroup of $\Gamma_{j_0}$ as in Definition \ref{def:suborbi-cover}. By the existence of $\Phi^\prime$ as in Proposition \ref{prop:strong-suborbifold} (and \cite[Lemma 2.11]{Moerdijk2003}), we can assume that $\Delta_{j_0}$ acts effectively on $\widetilde{V}_{j_0}$. Using that the isotropy of $[(0,0)]$ in $(Y,\mathcal{A})$ is $\mathbb{Z}_2$, we can now conclude that $\big\langle\begin{pmatrix}-1&0\\0&1\end{pmatrix}\big\rangle$ has to act effectively on an open neighborhood of $(0,0)$ in $\{0\}\times\mathbb{C}$ (compare the end of \cite[Example 5.4]{Weilandt2017}) -- a contradiction.
\end{example}

\begin{remark}
Examples \ref{ex:strong} (\ref{ex:strong:1}) and \ref{ex:nonstrong} suggest a close relation between strong and embedded suborbifolds. However, it is not clear to us if the two notions are equivalent in general.
\end{remark}

We shall now consider another type of suborbifold cover and the corresponding suborbifolds. In Theorem \ref{thm:very-strong} we will verify that the definition below indeed gives a special kind of strong suborbifold cover. 

\begin{definition}
\label{def:verystrong}
  Let $\mathcal{A}=\{(U_i,\widetilde{U}_i/\Gamma_i,\pi_i)\}_{i\in I}$ be a countable $n$-dimensional orbifold atlas on a second countable Hausdorff space $X$ and let $Y$ be a subset of $X$. A closed $k$-dimensional suborbifold cover $\{\widetilde{V}_j\}_{j\in J}$ on $Y$ with respect to $\mathcal{A}$ is called \emph{very strong} if it satisfies the following conditions:
  \begin{enumerate}  
  \item \label{def:verystrongU} 
  If $j\in J$ and $\gamma\in\Gamma_j$ fixes some open  subset of $\widetilde{V}_j$, then $\gamma=e$.
  \item \label{def:verystrongE} 
  For any $j_1,j_2\in J$ and $x\in\pi_{j_1}(\widetilde{V}_{j_1})\cap\pi_{j_2}(\widetilde{V}_{j_2})$, there is some orbifold chart $(W,\widetilde{W}/K,\sigma)$ on $X$, a subgroup $\Lambda\subset K$, a $k$-dimensional $\Lambda$-submanifold $\widetilde{Z}$ of $\widetilde{W}$ such that $\sigma(\widetilde{Z})$ is an open neighborhood of $x$ in $Y$ containing $x$ and there are injections $\lambda_m\co\widetilde{W}\to\widetilde{U}_{j_m}$, $m=1,2$, such that $\lambda_m(\widetilde{Z})\subset\widetilde{V}_{j_m}$ for every $m$.
  \end{enumerate}
\end{definition}

\begin{remark}
  \begin{enumerate}
  \label{rem:strong}
  \item Note that condition \eqref{def:verystrongU} above directly implies that every very strong cover is embedded.
  \item \label{rem:strong-cr} Readers familiar with \cite{Chen2002} may note that item \eqref{def:verystrongE} above is equivalent to the condition that the inclusions $\widetilde{V}_j\hookrightarrow\widetilde{U}_j$, $j\in J$, induce a ``$C^\infty$-lifting'' of the inclusion $Y\hookrightarrow X$ in the sense of \cite{Chen2002}.
  \end{enumerate}
\end{remark}

\begin{definition}
  Let $(X,\mathcal{A})$ be an orbifold. A subset $Y$ of $X$ is a \emph{very strong suborbifold} if there is a very strong suborbifold cover $\{\widetilde{V}_j\}_{j\in J}$ on $Y$ with respect to some countable atlas contained in $\mathcal{A}$.
\end{definition}

\begin{example}
\label{ex:very-strong}
  \begin{enumerate}
  \item Let $(X,\mathcal{A})$ be an $n$-dimensional orbifold and let $Y$ be an open subset of $X$. Then $Y$ is an $n$-dimensional very strong suborbifold of $(X,\mathcal{A})$: Given $x\in Y$, pick some chart $\{(U,\widetilde{U}/\Gamma,\pi)\}\in\mathcal{A}$ such that $x\in U$. Diminishing $U$ if necessary, we can assume that $U\subset Y$. Since $X$ is second countable, we can use this construction to obtain an atlas $\{(U_j,\widetilde{U}_j/\Gamma_j,\pi_j)\}_{j\in J}$ of $Y$ contained in a countable atlas $\mathcal{A}^\prime\subset\mathcal{A}$. In this special setting the suborbifold cover $\{\widetilde{U}_j\}_{j\in J}$ satisfies \eqref{def:verystrongE} simply due to the compatibility of  the corresponding charts in $\mathcal{A}^\prime$. Since every $\Gamma_j$ is finite and acts effectively on $\widetilde{U}_j$, the suborbifold cover satisfies \eqref{def:verystrongU} as well.
  \item \label{ex:very-strong:2} Note that condition \eqref{def:verystrongU} in Definition \ref{def:verystrong} directly implies that a singular point in an orbifold is not a very strong suborbifold (but a strong suborbifold by Example \ref{ex:discrete}).
  \end{enumerate}
\end{example}

\begin{theorem}
\label{thm:very-strong}
Every very strong suborbifold (cover) is strong.
\end{theorem}

\begin{proof}
Let $\mathcal{A}=\{(U_i,\widetilde{U}_i/\Gamma_i,\pi_i)\}_{i\in I}$ be a countable orbifold atlas on a second countable Hausdorff space $X$. If $\{\widetilde{V}_j\}_{j\in J}$ is a very strong suborbifold cover on a subset $Y\subset X$ with respect to $\mathcal{A}$ and $\mathcal{B}=\{(V_j,\widetilde{V}_j/\Delta_j,{\pi_j}_{|\widetilde{V}_j})\}_{j\in J}$ denotes the induced orbifold atlas on $Y$, then we have to show that there is a Lie groupoid embedding $\Phi\co\mathcal{G}(\mathcal{B})\to\mathcal{G}(\mathcal{A})$ such that $\varepsilon_{\mathcal{A}}\circ |\Phi|=\varepsilon_{\mathcal{B}}$.

  Let $H=\mathcal{G(B)}$ and $G=\mathcal{G(A)}$. First note that $H_0=\bigcup_{j\in J}\widetilde{V}_j\subset\bigcup_{i\in I}\widetilde{U}_i=G_0$. Now let $[f]_p\in H_1$ and let $j_1, j_2\in J$ such that $p\in \widetilde{V}_{j_1}$ and $f(p)\in\widetilde{V}_{j_2}$. Without loss of generality, assume $j_1=1$ and $j_2=2$. Diminishing the domain of $f$, we can assume that it is contained in $\widetilde{V}_1$. Since $\pi_2(f(p))=\pi_1(p)$, by condition \eqref{def:verystrongE} from Definition \ref{def:verystrong} there is an orbifold chart $(W,\widetilde{W}/K,\sigma)$ on $X$ a subgroup $\Lambda\subset K$ and a $\Lambda$-submanifold $\widetilde{Z}\subset\widetilde{W}$ such that $\pi_1(p)\in\sigma(\widetilde{Z})$ together with injections $\lambda_k\co\widetilde{W}\to\widetilde{U}_k$, $k=1,2$, such that each $\lambda_k(\widetilde{Z})$ is an open subset of $\widetilde{V}_k$. Diminishing $\widetilde{Z}$, $\widetilde{W}$ if necessary, we can assume that $\widetilde{Z}$ is closed in $\widetilde{W}$ and connected (see \cite[Lemma 2.6]{Weilandt2017}) and hence we obtain a chart $\sigma_{|\widetilde{Z}}$ on $Y$. Let $\gamma\in\Gamma_1$ such that $\gamma p\in\lambda_1(\widetilde{Z})$. Since $\widetilde{V}_1$ is a $\Delta_1$-submanifold, there is $\delta\in \Delta_1$ such that $\delta p=\gamma p$. Thus replacing $\lambda_1$ by $\delta^{-1}\circ\lambda_1$ if necessary, we can assume that $p\in\lambda_1(\widetilde{Z})$. Analogously, we can assume that $f(p)\in\lambda_2(\widetilde{Z})$. Since $f\circ{\lambda_1}_{|\widetilde{Z}}$ and ${\lambda_2}_{|\widetilde{Z}}$ are injections $\widetilde{Z}\to\widetilde{V}_2$ from $\sigma_{|\widetilde{Z}}$ to ${\pi_2}_{|\widetilde{V}_2}$, by \cite[Proposition A.1]{Moerdijk1997} we can, if necessary, replace $\lambda_2$ by a composition with an appropriate element of $\Delta_2$ to guarantee $f\circ{\lambda_1}_{|\widetilde{Z}}={\lambda_2}_{|\widetilde{Z}}$. Set $\Phi_1([f]_p)=[\lambda_2\circ\lambda_1^{-1}]_p\in G_1$. Condition \eqref{def:verystrongU} from Definition \ref{def:verystrong} implies that this extension of $[f]_p$ to an element of $G_1$ is unique. (Two such extensions would differ by some $\gamma\in\Gamma_2$ fixing the image of $f$. Since this image is open in $\widetilde{V}_2$, we conclude $\gamma=e$.) Therefore $\Phi_1$ is well-defined and injective. Since $\Phi_1$ locally corresponds to inclusions of the form $\dom f\hookrightarrow\lambda_1(\widetilde{W})$, it is an embedding.
  
  Together with the inclusion $\Phi_0\co H_0\hookrightarrow G_0$, we obtain a groupoid embedding $\Phi\co H\to G$. Since $\Phi_0$ is just the inclusion, we easily see that $\varepsilon_\mathcal{A}\circ|\Phi|=\varepsilon_\mathcal{B}$.
\end{proof}

\begin{remark}
  The proof above shows that if $\mathcal{A}$ is a countable atlas on a second countable Hausdorff space $X$, $\{\widetilde{V}_j\}_{j\in J}$ is a very strong suborbifold cover on $Y\subset X$ with respect to $\mathcal{A}$ and $\mathcal{B}$ is the induced orbifold atlas on $Y$, then there is a unique Lie groupoid embedding $\Phi\co\mathcal{G}(\mathcal{B})\to\mathcal{G}(\mathcal{A})$ such that $\Phi_0$ is the inclusion.
\end{remark}

\begin{definition}
\label{def:regular-cover}
  Let $\mathcal{A}=\{(U_i,\widetilde{U}_i/\Gamma_i,\pi_i)\}_{i\in I}$ be an $n$-dimensional orbifold atlas on a second countable Hausdorff space $X$. A \emph{regular suborbifold cover} on a subset $Y\subset X$ is a closed suborbifold cover $\{\widetilde{V}_j\}_{j\in J}$ on $Y$ such that every
 $\{p\in\widetilde{V}_j;~{\Gamma_j}_p=\{e\}\}$ is dense in $\widetilde{V}_j$ and connected.
\end{definition}

\begin{remark}
  Note that the set $\{p\in\widetilde{V}_j;~{\Gamma_j}_p=\{e\}\}$ considered above is automatically open, as it is the intersection of the set of points with trivial isotropy in $\widetilde{U}_j$ with $\widetilde{V}_j$.
\end{remark}

\begin{definition}
\label{def:regular-suborbifold}
  Let $(X,\mathcal{A})$ be an orbifold. A subset $Y$ of $X$ is a \emph{regular suborbifold} if there is a regular suborbifold cover $\{\widetilde{V}_j\}_{j\in J}$ on $Y$ with respect to some countable atlas contained in $\mathcal{A}$.
\end{definition}

\begin{remark}
  Our definitions above are inspired by the definition of a ``regular'' map in \cite[Definition 4.4.10]{Chen2002}: Remark \ref{rem:strong} \eqref{rem:strong-cr} and Proposition \ref{prop:regular} below imply that given a regular cover $\{\widetilde{V}_j\}_{j\in J}$ in the sense of Definition \ref{def:regular-cover}, the inclusions $\widetilde{V}_j\hookrightarrow\widetilde{U}_j$, $j\in J$, define a ``regular $C^\infty$'' lifting (in the sense of \cite{Chen2002}) of the inclusion. However, the converse does not hold: If $\rho$ denotes the rotation of $\mathbb{R}^2$ around the origin by $\pi/2$, the inclusion of $\mathbb{R}\times\{0\}$ into $\mathbb{R}^2$ gives a ``regular $C^\infty$'' lifting (\cite{Chen2002}) between the orbifolds $(\mathbb{R}\times\{0\})/\langle\rho^2\rangle$ and $\mathbb{R}^2/\langle\rho\rangle$ but the suborbifold is not regular in the sense of Definition \ref{def:regular-suborbifold}.
\end{remark}

Before verifying that every regular suborbifold is very strong, we will consider certain graphs as a basic example. In order to talk about graphs we will need the following rather simple notion of a smooth map, which is sufficient to guarantee well-behaved graphs and to guarantee that embedded suborbifolds are images of embeddings (\cite{Weilandt2017}). (We just consider certain continuous maps between the underlying spaces in the spirit of \cite{Satake1956}, other sources including \cite{Satake1957,Chen2002} consider equivalence relations on collections of smooth local lifts or even more complex structures (\cite{Pohl2017}).)

\begin{definition}
Given two orbifolds $(X,\mathcal{A})$, $(X^\prime,\mathcal{A}^\prime)$, a \emph{smooth map} $f\co (X,\mathcal{A})\to (X^\prime,\mathcal{A}^\prime)$ is a continuous map $f\co X\to X^\prime$ such that for each $x\in X$ there is a chart $(U,\widetilde{U}/\Gamma,\pi)\in\mathcal{A}$ around $x$, a chart $(U^\prime,\widetilde{U}^\prime/\Gamma^\prime,\pi^\prime)\in\mathcal{A}^\prime$ around $f(x)$, a smooth map $\widetilde{f}\co \widetilde{U}\to\widetilde{U}^\prime$ and a homomorphism $\overline{f}\co \Gamma\to\Gamma^\prime$ such that $\pi^\prime\circ \widetilde{f}=f\circ\pi$ and $\widetilde{f}(\gamma p)=\overline{f}(\gamma)\widetilde{f}(p)$ for every $\gamma\in\Gamma$, $p\in\widetilde{U}$.
\end{definition}

Given any map $f$, we will denote its graph by $\gr f$. Recall from \cite[Proposition 3.7]{Weilandt2017} that the graph of a smooth map between orbifolds is an embedded suborbifold and note that the given proof shows that $\gr f$ is fully embedded if all points in the image of $f$ are regular. Modifying the latter condition gives a criterion for a graph to be a regular suborbifold:

\begin{proposition}
  Let $(X,\mathcal{A})$ and $(X^\prime,\mathcal{A}^\prime)$ be orbifolds such that the singular strata of $(X,\mathcal{A})$ have codimension at least two. If $f\co (X,\mathcal{A})\to (X^\prime,\mathcal{A}^\prime)$ is a smooth map such that $f((X,\mathcal{A})^\reg)\subset(X^\prime,\mathcal{A}^\prime)^\reg$, then $\gr f$ is a regular suborbifold of $(X\times X^\prime,\mathcal{A}\times\mathcal{A}^\prime)$.
\end{proposition}

\begin{proof}
  Let $\{(U_i,\widetilde{U}_i/\Gamma_i,\pi_i)\}_{i\in I}\subset\mathcal{A}$ be a countable atlas on $X$ and consider charts $\{(U_i ^\prime,\widetilde{U}^\prime_i/\Gamma^\prime_i,\pi^\prime_i)\}_{i\in I}\subset\mathcal{A}^\prime$ on $X^\prime$ such that for every $i\in I$ there is a smooth map $\widetilde{f}_i\co\widetilde{U}_i\to\widetilde{U}_i^\prime$ and a homomorphism $\overline{f}_i\co\Gamma_i\to\Gamma_i^\prime$ such that $\pi_i^\prime\circ\widetilde{f}_i=f_i\circ\pi_i$ and $\widetilde{f}_i(\gamma p)=\overline{f}_i(\gamma)\widetilde{f}_i(p)$. Completing $\{(U_i ^\prime,\widetilde{U}^\prime_i/\Gamma^\prime_i,\pi^\prime_i)\}$ to a countable atlas $\{(U_j^\prime,\widetilde{U}^\prime_j/\Gamma^\prime_j,\pi^\prime_j)\}_{j\in J}\subset\mathcal{A}^\prime$, we obtain a countable atlas \[\mathcal{C}=\{(U_i\times U_j^\prime,(\widetilde{U}_i\times\widetilde{U}_j^\prime)/(\Gamma_i\times\Gamma_j^\prime),\pi_i\times\pi_j^\prime)\}_{(i,j)\in I\times J}\subset\mathcal{A}\times\mathcal{A}^\prime.\]
  
  For each $i\in I$ set $\widetilde{V}_i=\gr\widetilde{f}_i\subset\widetilde{U}_i\times\widetilde{U}_i^\prime$ and $\Delta_i=\gr\overline{f}_i\subset \Gamma_i\times\Gamma_i^\prime$ and note that $\widetilde{V}_i$ is a closed $\Delta_i$-submanifold such that $(\pi_i\times\pi_i^\prime)(\widetilde{V}_i)=\gr f\cap (U_i\times U_i^\prime)$ is open in $\gr f$. Since the union of the latter sets cover $\gr f$, the family $\{\widetilde{V}_i\}_{i\in I}$ is a closed suborbifold cover of $\gr f$ with respect to $\mathcal{C}$.
  
  To see that this suborbifold cover is regular, we have to show that for each $i\in I$ the set $S=\{(p,\widetilde{f}_i(p));~p\in\widetilde{U}_i,~(\Gamma_i\times\Gamma_i^\prime)_{(p,\widetilde{f}_i(p))}=\{e\}\}$ is dense in $\widetilde{V}_i$ and connected. Note that the regular part $\widetilde{U}_i^\reg=\{(p\in\widetilde{U}_i;~{\Gamma_i}_p=\{e\}\}$ is dense in $\widetilde{U}_i$ and connected due to the codimension condition. Hence the graph of the restriction of ${\widetilde{f}_i}$ to $\widetilde{U}_i^\reg$ is dense in $\widetilde{V}_i$ and connected. Moreover, it is contained in $S$, since $f((X,\mathcal{A}))^\reg\subset(X^\prime,\mathcal{A}^\prime)^\reg$. We conclude that $S$ is also dense in $\widetilde{V}_i$ and connected.
\end{proof}

\begin{example}
\begin{enumerate}
\item Applying the proposition above to the identity on an orbifold $(X,\mathcal{A})$ whose singular stratum has codimension at least two, we observe that the diagonal $D=\{(x,x);~x\in X\}$ is a very strong suborbifold of $(X\times X,\mathcal{A}\times\mathcal{A})$. (Note that the diagonal is often taken as a litmus test if a suborbifold definition is sufficiently general (\cite{Borzellino2008}, also compare \cite{Cho2013} for the groupoid setting). For instance, a simple isotropy argument shows that the diagonal is not fully embedded if the orbifold contains a singular point.)

\item We should note that a graph of an arbitrary smooth map need not be a very strong suborbifold. Consider, for instance, a map $\{\ast\}\to (X,\mathcal{A})$ whose domain is a single point and whose image is a singular point (compare Example \ref{ex:very-strong} \eqref{ex:very-strong:2}).
\end{enumerate}
\end{example}

To verify that every regular cover is very strong, we will need the following lemma.

\begin{lemma}
\label{lemma:transf}
  Let $M$ be a manifold, let $\Gamma$ be a finite group acting smoothly and effectively on $M$ and let $\pi\co M\to M/\Gamma$ denote the canonical projection. Let $N_1$ be a submanifold of $M$ and let $N_2$ be a closed $\Delta$-submanifold of $M$ with respect to some subgroup $\Delta$ of $\Gamma$. If $\pi(N_1)\subset\pi(N_2)$ and $N_1^\prime:=\{p\in N_1;~\Gamma_p=\{e\}\}$ is connected and dense in $N_1$, then there is $\gamma\in \Gamma$ such that $\gamma N_1\subset N_2$.
\end{lemma}

\begin{proof}
  First note that $N_1^\prime$ is open in $N_1$.
  
  Given $p\in N_1^\prime$, we shall show that there is an open neighborhood $U_p\in N_1^\prime$ of $p$ and $\gamma_p\in \Gamma$ such that $\gamma_pU_p\subset N_2$: Since $\pi(N_1)\subset\pi(N_2)$, there is $\gamma_p$ such that $\gamma_pp\in N_2$. If there was no neighborhood $U_p$ as desired, there would be a sequence $(p_n)\subset N_1$ converging to $p$ such that $\gamma_pp_n\notin N_2$ for every $n\in\mathbb{N}$. Let $(\gamma_n)\subset \Gamma$ be a sequence such that $\gamma_np_n\in N_2$. Since $\Gamma$ is finite, we can assume that $(\gamma_n)$ is constant, say $\gamma_0$. Since $N_2$ is closed, we obtain $\gamma_0p=\gamma_0\lim p_n=\lim(\gamma_0p_n)\in N_2$ and $\gamma_pp\in N_2$ implies that there is $\delta\in \Delta$ such that $\delta \gamma_0p=\gamma_pp$. Since $p\in N_1^\prime$, we obtain $\gamma_p^{-1}\delta\gamma_0\in\{e\}$ and conclude that $\gamma_pp_n=\delta\gamma_0p_n\in N_2$ -- a contradiction.
  
Fix $p_0\in N_1^\prime$ and write $\gamma=\gamma_{p_0}$. To see that $\gamma N_1^\prime\subset N_2$, let $q_0\in N_1^\prime$. For each $p\in N_1^\prime$ let $\gamma_p\in \Gamma$ and $U_p\subset N_1$ be as in the preceding paragraph. Being a connected submanifold of $N_1$, $N_1^\prime$ is path-connected. Let $c\co [0,1]\to N_1^\prime$ be a smooth curve such that $c(0)=p_0,c(1)=q_0$. Let $p_0,\ldots,p_k=q_0$ be points in the image $\im(c)$ of $c$ such that $\im(c)\subset \bigcup_{i=0}^kU_{p_i}$ and $U_{p_i}\cap U_{p_{i+1}}\neq\emptyset$ for each $i$. Given $r_i\in U_{p_i}\cap U_{p_{i+1}}$, we have $\gamma_{p_i} r_i,\gamma_{p_{i+1}}r_i\in N_2$. Since $N_2$ is a $\Delta$-submanifold and $\Gamma_{r_i}=\{e\}$, we obtain $\gamma_{p_i}\gamma_{p_{i+1}}^{-1}\in \Delta$. Since this holds for every $i$, we can write
\[\gamma U_{q_0}=(\gamma_{p_0}\gamma_{p_1}^{-1})(\gamma_{p_1}\gamma_{p_2}^{-1})\cdots (\gamma_{p_{k-1}}\gamma_{p_k}^{-1})\gamma_{p_k}U_{p_k}\subset N_2.\]
In particular, $\gamma q_0\in N_2$.
  
  Since $N_1^\prime$ is dense in $N_1$ and $N_2$ is closed, we conclude that $\gamma N_1\subset N_2$.
\end{proof}

\begin{remark}
  Note that, in the setting of the lemma above, it is not sufficient to demand $\pi(N_1^\prime)$ (instead of $N_1^\prime$ itself) to be connected and dense in $\pi(N_1)$: Consider $M=\mathbb{R}^2$. Let $f\co \mathbb{R}\to\mathbb{R}$ be a strictly increasing odd function such that $f^{(k)}(0)=0$ for all $k\ge 0$, let $N_1$ be its graph and let $N_2$ be the graph of $|f|$. Let $\rho$ denote the rotation of $\mathbb{R}^2$ by $\pi/2$ and let $s$ denote the reflection along the $y$-axis. Let $D_4$ denote the dihedral group generated by $\rho$ and $s$ and set $\Gamma=D_4$, $\Delta=\langle s\rangle\simeq\mathbb{Z}_2$. Then $N_1, N_2$ satisfy almost all conditions from Lemma \ref{lemma:transf}, the only exception being that $N_1^\prime=N_1\setminus\{(0,0)\}$ is disconnected (but $\pi(N_1^\prime)$ is connected). It is easy to see that there is no $\gamma$ as in the conclusion of the lemma.
\end{remark}

\begin{proposition}
\label{prop:regular}
  Let $\mathcal{A}=\{(U_i,\widetilde{U}_i/\Gamma_i,\pi_i)\}_{i\in I}$ be a countable $n$-dimensional orbifold atlas on a second countable Hausdorff space $X$. If $\{\widetilde{V}_j\}_{j\in J}$ is a regular suborbifold cover on a subset $Y\subset X$, then it is very strong.
\end{proposition}

\begin{proof}
  To verify condition \eqref{def:verystrongU} from Definition \ref{def:verystrong}, let $j\in J$ and let $\gamma\in\Gamma_j$ such that $\gamma_{|S}=\text{id}_{S}$ for some open $S \subset \widetilde{V}_j$. Since our cover is regular, there is $p\in S$ such that ${\Gamma_j}_p=\{e\}$. Hence $\gamma=e$.
  
  To verify condition \eqref{def:verystrongE}, let $j_1,j_2\in J$ and $x\in\pi_{j_1}(\widetilde{V}_{j_1})\cap\pi_{j_2}(\widetilde{V}_{j_2})$. Assume $j_1=1$, $j_2=2$ and write $V_{m}:=\pi_m(\widetilde{V}_m)$ for $m=1,2$. By the compatibility of $\pi_1$ and $\pi_2$, there is an orbifold chart $(W,\widetilde{W}/K,\sigma)$ on $X$ around $x$ with injections $\lambda_m\co \widetilde{W}\to\widetilde{U}_m$. Diminishing $W$ if necessary, we can assume that $W\cap Y\subset V_1\cap V_2$. Let $\widetilde{x}\in \sigma^{-1}(x)$, let $\widetilde{Z}$ be the connected component of $\lambda_1^{-1}(\widetilde{V}_1)$ containing $\widetilde{x}$ and set
 $\Lambda:=\{k\in\overline\lambda_1^{-1}(\Delta_1);~k\widetilde{Z}\subset\widetilde{Z}\}$. Then $\widetilde{Z}$ is a connected $\Lambda$-submanifold of $\widetilde{W}$ and $\sigma(\widetilde{Z})$ is open in $\sigma(\lambda_1^{-1}(\widetilde{V}_1))=V_1\cap W$ and hence an open neighborhood of $x$ in $Y$. Regularity of the cover implies that $\lambda_2(\widetilde{Z})^\prime=\{p\in\lambda_2(\widetilde{Z});~{\Gamma_2}_p=\{e\}\}$ is connected and dense in $\lambda_2(\widetilde{Z})$. Since $\pi_2(\lambda_2(\widetilde{Z}))=\sigma(\widetilde{Z})\subset W\cap Y\subset V_2$ and $\widetilde{V}_2$ is closed in $\widetilde{U}_2$, by Lemma \ref{lemma:transf}, there is $\gamma\in\Gamma_2$ such that $\gamma\circ\lambda_2(\widetilde{Z})\subset\widetilde{V}_2$. Replacing $\lambda_2$ by $\gamma\circ\lambda_2$, we finally obtain injections $\lambda_1$, $\lambda_2$ as in condition \eqref{def:verystrongE} of Definition \ref{def:verystrong}.
\end{proof}

\begin{example}
  If $(X,\mathcal{A})$ is an orbifold and $Y$ is a fully embedded suborbifold whose singular stratum has codimension at least $2$ (in $(Y,\mathcal{A}_{|Y})$), then $Y$ is a regular suborbifold.
\end{example}

\subsection*{Acknowledgments} \ \

\noindent The work of J.N.M. is supported by the Max-Planck-Gesellschaft. Both authors would like to thank the anonymous referee for carefully reading the paper and suggesting various improvements, including the legitimate demand for an example of a suborbifold which is not strong.

\bibliographystyle{plain+url}
\bibliography{lit3}

\end{document}